\documentclass[12pt,a4paper]{article}
\usepackage[latin 2]{inputenc}
\usepackage{amssymb}
\usepackage{latexsym}
\usepackage{amsmath}
\usepackage{amsthm}
\newtheorem{thm}{Theorem}
\newtheorem{prop}[thm]{Proposition}

\theoremstyle{definition}
\usepackage{xpatch}
\xpatchcmd{\proof}{\itshape}{\normalfont\proofnameformat}{}{}
\newcommand{\proofnameformat}{}
\begin{document}

\renewcommand{\proofnameformat}{\bfseries}

\begin{center}
{\large\textbf{On the distribution of the van der Corput sequence in arbitrary base}}

\vspace{5mm}

\textbf{Bence Borda}

{\footnotesize Department of Mathematics, Rutgers University

110 Frelinghuysen Road, Piscataway, NJ-08854, USA

Email: \texttt{bordabence85@gmail.com}}

\vspace{5mm}


\end{center}

\vspace{5mm}

\begin{abstract}
A central limit theorem with explicit error bound, and a large deviation result are proved for a sequence of weakly dependent random variables of a special form. As a corollary, under certain conditions on the function $f: [0,1] \to \mathbb{R}$ a central limit theorem and a large deviation result are obtained for the sum $\sum_{n=0}^{N-1} f(x_n)$, where $x_n$ is the base $b$ van der Corput sequence for an arbitrary integer $b \ge 2$. Similar results are also proved for the $L^p$ discrepancy of the same sequence for $1 \le p < \infty$. The main methods used in the proofs are the Berry--Esseen theorem and Fourier analysis.
\end{abstract}

\section{Introduction}
\label{Introduction}

For an integer $b \ge 2$ the base $b$ van der Corput sequence $x_n$ is defined the following way. If the base $b$ representation of the integer $n \ge 0$ is $n = \sum_{i=1}^m a_i b^{i-1}$ for some digits $a_i \in \left\{ 0 , 1 , \dots , b-1 \right\}$, then

\[ x_n = \sum_{i=1}^m \frac{a_i}{b^i} . \]

\noindent The main importance of this sequence is that it is of low discrepancy. Indeed, the discrepancy function of the base $b$ van der Corput sequence

\[ \Delta_N (x) = \left| \left\{ 0 \le n < N : x_n < x \right\} \right| -Nx , \]

\noindent defined for nonnegative integers $N$, and $x \in [0,1]$, satisfies

\[ 0 \le \Delta_N (x) \le \frac{b}{4} \log_b N +b. \]

\newpage

\noindent The precise value of

\[ \limsup_{N \to \infty} \frac{\sup_{x \in [0,1]} \Delta_N (x) }{\log N} \]

\noindent in terms of the base $b$ was found by Faure (\cite{4} Theorem 1, Theorem 2 and Sections 5.5.1--5.5.3).

In this article we study the random aspects of the base $b$ van der Corput sequence. Let

\[ \Phi (\lambda ) = \int_{- \infty}^{\lambda} \frac{1}{\sqrt{2 \pi}} e^{- \frac{x^2}{2}} \, \mathrm{d}x \]

\noindent denote the distribution function of the standard normal distribution. Our main result is that  the sum

\[ S(N) = \sum_{n=0}^{N-1} \left( \frac{1}{2} - x_n \right) \]

\noindent satisfies the following central limit theorem.

\begin{thm}\label{Theorem1} Let $x_n$ be the base $b$ van der Corput sequence, where $b \ge 2$ is an arbitrary integer. Then for any integer $M > b^2$ and any real number $\lambda$ we have

\[ \frac{1}{M} \left| \left\{ 0 \le N < M : \frac{S(N) - c(b) \log_b N}{\sqrt{d(b) \log_b N}} < \lambda  \right\} \right| = \Phi ( \lambda ) + O \left( \frac{\sqrt[4]{\log \log_b M}}{\sqrt[4]{\log_b M}} \right) , \]

\noindent where $\displaystyle{c(b) = \frac{b^2-1}{12b}}$ and $\displaystyle{d(b) = \frac{b^4+120 b^3 -480 b^2 + 600 b -241}{720 b^2}}$. The implied constant in the error term is absolute.
\end{thm}

\noindent The following large deviation result complements Theorem 1.

\begin{thm}\label{Theorem2} Let $x_n$ be the base $b$ van der Corput sequence, where $b \ge 2$ is an arbitrary integer. For any integer $M > b$ and any real number $\lambda \ge 3$ we have

\begin{multline*}
\frac{1}{M} \left| \left\{ 0 \le N < M : \left| S(N) - \frac{b^2-1}{12 b} \log_b M \right| \ge 25 \lambda b \sqrt{\log_b M +1} \right\} \right| \\ \le \frac{4 \sqrt{\lambda}}{e^{\sqrt{\lambda}-1}-2} + \frac{1}{b^{\sqrt{\log_b M} -2}} .
\end{multline*}
\end{thm}

\noindent Since

\begin{equation}\label{SN}
\int_0^1 \Delta_N (x) \, \mathrm{d}x = S(N) ,
\end{equation}

\noindent we have that $S(N) = O \left( b \log_b M \right)$, therefore Theorem \ref{Theorem2} is meaningful only when applied with $\lambda = O \left( \sqrt{\log_b M} \right)$. Note that for all such values of $\lambda$ the error term $\frac{1}{b^{\sqrt{\log_b M} -2}}$ is of smaller order of magnitude than $\frac{4 \sqrt{\lambda}}{e^{\sqrt{\lambda}-1}-2}$. The question of whether the upper bound in Theorem \ref{Theorem2} can be improved to $O \left( e^{-d \lambda} \right)$ or to $O \left( e^{-d \lambda^2} \right)$ for some constant $d>0$ is left open.

Observation (\ref{SN}) gives the idea that the sum $S(N)$ is related to the $L^p$ norm

\[ \left\| \Delta_N \right\|_p = \left( \int_0^1 \left| \Delta_N (x) \right|^p \, \mathrm{d}x \right)^{\frac{1}{p}} \]

\noindent of the discrepancy function. As simple corollaries to Theorem \ref{Theorem1} and Theorem \ref{Theorem2} we thus obtain that $\left\| \Delta_N \right\|_p$ satisfies the same central limit theorem and large deviation result as $S(N)$.

\begin{thm}\label{Theorem3} Let $x_n$ be the base $b$ van der Corput sequence, where $b \ge 2$ is an arbitrary integer. Let $1 \le p < \infty$ be an arbitrary real. Then for any integer $M > b^2$ and any real number $\lambda$ we have

\[ \frac{1}{M} \left| \left\{ 0 \le N < M : \frac{\left\| \Delta_N \right\|_p - c(b) \log_b N}{\sqrt{d(b) \log_b N}} < \lambda  \right\} \right| = \Phi ( \lambda ) + O \left( \frac{\sqrt[4]{\log \log_b M}}{\sqrt[4]{\log_b M}} \right) , \]

\noindent where $\displaystyle{c(b) = \frac{b^2-1}{12b}}$ and $\displaystyle{d(b) = \frac{b^4+120 b^3 -480 b^2 + 600 b -241}{720 b^2}}$. The implied constant in the error term depends only on $p$.
\end{thm}

\begin{thm}\label{Theorem4} Let $x_n$ be the base $b$ van der Corput sequence, where $b \ge 2$ is an arbitrary integer. Let $1 \le p < \infty$ be an arbitrary real. There exists a positive constant $A$ depending only on $p$ such that for any integer $M > b$ and any real number $\lambda \ge 1$ we have

\[ \frac{1}{M} \left| \left\{ 0 \le N < M : \left| \left\| \Delta_N \right\|_p - \frac{b^2-1}{12 b} \log_b N \right| \ge A \lambda b \sqrt{\log_b N} \right\} \right| \le e^{- \sqrt{\lambda}} . \]

\end{thm}

Similar central limit theorems concerning the distribution of the van der Corput sequence have already appeared in the literature. In \cite{3} Theorem \ref{Theorem3} is proved in the special case when $b=2$ with an error term $o (1)$ of unspecified order of magnitude. In Section 1.3 of \cite{1} Theorem \ref{Theorem1} is proved, again in the special case $b=2$, with an error term $\displaystyle{O \left( \frac{\log \log M}{\sqrt[10]{\log M}} \right)}$. Our proof of Theorem \ref{Theorem1} is the generalization of the proof in Section 1.3 of \cite{1}. In a doctoral dissertation (\cite{phd} Theorem 4.1.1.) a central limit theorem for the supremum norm $\left\| \Delta_N \right\|_{\infty}$ of the discrepancy function in the case of an arbitrary base $b \ge 2$, similar to Theorem \ref{Theorem3} is proved. The main difference is that $c(b)$ is to be replaced by $c_{\infty}(b)=\frac{2b-1}{12}$ and $d(b)$ is to be replaced by

\[ d_{\infty} (b) = \frac{4b^7 - 10b^6 + 10b^5 + 14b^4 - 77b^3 + 127b^2 - 68b +8}{720 b^2 (b-1)^2(b+1)} . \]

\noindent Moreover, the theorem is stated only in the special case when $M$ is a power of the base $b$, and the error term is of an unspecified order of magnitude $o(1)$. In \cite{3} and \cite{phd} central limit theorems for various generalizations of the van der Corput sequence are also studied. Large deviation results have not yet been obtained.

Finally, we give a method to generalize Theorem \ref{Theorem1} and Theorem \ref{Theorem2} for sums of the form $\sum_{n=0}^{N-1} f(x_n)$, where the function $f:[0,1] \to \mathbb{R}$ is sufficiently nice, and $x_n$ is the base $b$ van der Corput sequence. Since the discrepancy satisfies

\[ \sup_{x \in [0,1]} \left| \Delta_N (x) \right| = O \left( b \log_b N \right) , \]

\noindent the Koksma inequality (\cite{5} Chapter 2 Theorem 5.1) implies that if $f : [0,1] \to \mathbb{R}$ is of bounded variation, then

\[ \sum_{n=0}^{N-1} f(x_n) = N \int_0^1 f(x) \, \mathrm{d}x + O \left( \log N \right) , \]

\noindent as $N \to \infty$, with an implied constant depending only on $b$ and the total variation of $f$. Under more restrictive assumptions on the function $f$ the error term actually satisfies a central limit theorem and a large deviation result. The following proposition reduces the problem of studying the distribution of $\sum_{n=0}^{N-1} f(x_n)$ to that of $S(N)$.

\begin{prop}\label{Proposition5} Let $f:[0,1] \to \mathbb{R}$ be twice differentiable with $f'' \in L^1([0,1])$, and let $x_n$ denote the base $b$ van der Corput sequence, where $b \ge 2$ is an arbitrary integer. For any integer $N>0$ we have

\[ \left| \sum_{n=0}^{N-1} f(x_n) - N \int_0^1 f(x) \, \mathrm{d}x + \left( f(1) - f(0) \right) S(N) \right| \le \frac{b}{3} \left\| f'' \right\|_1 . \]

\end{prop}

\noindent The natural interpretation of the quantity $f(1) - f(0)$ is that the periodic extension of $f$ on $\mathbb{R}$ with period 1 has jumps of this size.

In Section \ref{section2} we derive the normalizing factors $c(b)$ and $d(b)$ of Theorem \ref{Theorem1}. Section \ref{section3} is devoted to the proofs of Theorem \ref{Theorem1} and Theorem \ref{Theorem2}, while the proofs of Theorem \ref{Theorem3}, Theorem \ref{Theorem4} and Proposition \ref{Proposition5} are given in Section \ref{section4}.

\section{The expected value and the variance of $S(N)$}\label{section2}

We start by deriving a formula for the sum $S(N)$ in terms of the base $b$ digits of $N$ as follows.

\begin{prop}\label{Proposition6} Let $b \ge 2$ be an integer and let $N = \sum_{i=1}^{m} a_i b^{i-1}$ be the base $b$ representation of an integer $N \ge 0$, where $a_i \in \left\{ 0 , 1 , \dots , b-1 \right\}$. Then

\[ S(N) = \sum_{i=1}^m \frac{(b+1) a_i - a_i^2}{2b} - \sum_{1 \le i < j \le m} \frac{a_i a_j}{b^{j-i+1}} . \]

\end{prop}

\begin{proof} By splitting the sum $S(N)$ we get

\begin{equation}\label{split}
S(N) = \sum_{n=0}^{a_m b^{m-1}-1} \left( \frac{1}{2} - x_n \right) + \sum_{n=a_m b^{m-1}}^{N-1} \left( \frac{1}{2} - x_n \right) .
\end{equation}

\noindent Since

\[ \left\{ x_n : 0 \le n < a_m b^{m-1} \right\} = \left\{ \frac{k}{b^{m-1}} + \frac{a}{b^m} : 0 \le k < b^{m-1} , \quad 0 \le a < a_m \right\} , \]

\noindent we obtain that the first sum in (\ref{split}) is

\[ \sum_{n=0}^{a_m b^{m-1}-1} \left( \frac{1}{2} - x_n \right) = \sum_{k=0}^{b^{m-1}-1} \sum_{a=0}^{a_m-1} \left( \frac{1}{2} - \frac{k}{b^{m-1}} - \frac{a}{b^m} \right) = \frac{(b+1) a_m - a_m^2}{2b} . \]

\noindent To compute the second sum in (\ref{split}) note that for any $a_m b^{m-1} \le n < N$ the first base $b$ digit of $n$ is $a_m$, and hence

\[ x_n = x_{n-a_m b^{m-1}} + \frac{a_m}{b^m} . \]

\noindent Therefore by reindexing the sum we obtain

\[ \begin{split}
\sum_{n=a_m b^{m-1}}^{N-1} \left( \frac{1}{2} - x_n \right) & = \sum_{n=0}^{N-a_m b^{m-1}-1} \left( \frac{1}{2} - x_{n} - \frac{a_m}{b^{m}} \right) \\
& = S(N-a_m b^{m-1}) - \frac{a_m}{b^m} (N-a_m b^{m-1}) .
\end{split} \]

\noindent Using the base $b$ representation of $N$ we thus find the recursion

\begin{equation}\label{recursion}
S \left( \sum_{i=1}^m a_i b^{i-1} \right) = \frac{(b+1)a_m - a_m^2}{2b} - \sum_{i=1}^{m-1} \frac{a_i a_m}{b^{m-i+1}} + S \left( \sum_{i=1}^{m-1} a_i b^{i-1} \right) .
\end{equation}

\noindent Applying the recursion (\ref{recursion}) $m$ times finishes the proof.

\end{proof}

If $N$ is a random variable uniformly distributed in $\left\{ 0 , 1 , \dots , b^{m}-1 \right\}$ for some integers $b \ge 2$ and $m \ge 1$, then the base $b$ digits $a_1 , \dots , a_m$ of $N$ are independent random variables, each uniformly distributed in $\left\{ 0 , 1 , \dots , b-1 \right\}$. Therefore Proposition \ref{Proposition6} can be used to find the expected value and the variance of the sum $S(N)$. Here and from now on the expected value and the variance of a real valued random variable $X$ are denoted by $\mathrm{E} \, (X)$ and $\mathrm{Var} \, (X)$, respectively.

\begin{prop}\label{Proposition7} Let $N$ be a random variable which is uniformly distributed in $\left\{ 0 , 1 , \dots , b^{m}-1 \right\}$ for some integers $b \ge 2$ and $m \ge 1$. Then

\[ \left| \mathrm{E} \, \left( S(N) \right) - \frac{b^2-1}{12b} m \right| \le \frac{1}{4} , \]

\[ \mathrm{Var} \, \left( S(N) \right) = \frac{b^4+120 b^3 -480b^2 +600b - 241}{720 b^2} m + O(b) . \]

\noindent The implied constant in the error term is absolute.
\end{prop}

\begin{proof} Using the independence of the base $b$ digits $a_1 , \dots , a_m$ of $N$, from Proposition \ref{Proposition6} we get that the expected value of $S(N)$ is

\[ \mathrm{E} \, \left( S(N) \right) = \sum_{i=1}^m \frac{(b+1) \mathrm{E} \, (a_i) - \mathrm{E} \, (a_i^2)}{2b} - \sum_{1 \le i < j \le m} \frac{\mathrm{E} \, (a_i) \mathrm{E} \, (a_j) }{b^{j-i+1}} = \frac{b^2-1}{12b} m + \frac{1}{4} - \frac{1}{4 b^m} . \]

To find the variance of $S(N)$, first let us use the independence of $a_1 , \dots , a_m$ again to obtain

\begin{equation}\label{variance}
\mathrm{Var} \, \left( \sum_{i=1}^m \frac{(b+1) a_i - a_i^2}{2b} \right) = \sum_{i=1}^m \mathrm{Var} \, \left( \frac{(b+1) a_i - a_i^2}{2b} \right) = \frac{b^4 + 55 b^2 -56}{720 b^2} m .
\end{equation}

\noindent Now consider

\begin{equation}\label{variance2}
\mathrm{Var} \, \left( \sum_{1 \le i < j \le m} \frac{a_i a_j}{b^{j-i+1}} \right) = \sum_{\substack{1 \le i_1 < j_1 \le m \\ 1 \le i_2 < j_2 \le m}} \left( \mathrm{E} \, \left( a_{i_1} a_{j_1} a_{i_2} a_{j_2} \right) - \frac{(b-1)^4}{16} \right) \frac{1}{b^{j_1 - i_1 +1} b^{j_2-i_2+1}} .
\end{equation}

\noindent We will group the terms according to the size of $\left\{ i_1 , j_1 \right\} \cap \left\{ i_2 , j_2 \right\}$. If $\left\{ i_1 , j_1 \right\} \cap \left\{ i_2 , j_2 \right\}$ is the empty set, then $a_{i_1}, a_{j_1}, a_{i_2}, a_{j_2}$ are independent, and therefore the contribution is zero.

If $\left\{ i_1 , j_1 \right\} \cap \left\{ i_2 , j_2 \right\}$ has size 1, then

\[ \mathrm{E} \, \left( a_{i_1} a_{j_1} a_{i_2} a_{j_2} \right) - \frac{(b-1)^4}{16} = \frac{(b-1)^3 (b+1)}{48} . \]

\noindent Let $s>0$, $t>0$ and $1 \le A \le m-s-t$ be integers. The sum of $\frac{1}{b^{j_1-i_1+1} b^{j_2-i_2+1}}$ over all $1 \le i_1 < j_1 \le m$ and $1 \le i_2 < j_2 \le m$ such that $\left\{ i_1 , j_1 \right\} \cup \left\{ i_2 , j_2 \right\} = \left\{ A , A+s, A+s+t \right\}$ is $\frac{2}{b^{2s+t+2}} + \frac{2}{b^{s+t+2}} + \frac{2}{b^{s+2t+2}}$, hence we have that the contribution of this case in (\ref{variance2}) is

\begin{multline*}
\frac{(b-1)^3 (b+1)}{48} \sum_{\substack{s,t>0 \\ s+t \le m}} \left( m-s-t \right) \left( \frac{2}{b^{2s+t+2}} + \frac{2}{b^{s+t+2}} + \frac{2}{b^{s+2t+2}} \right) \\ = \frac{b^2+2b-3}{24b^2} m + O(1) .
\end{multline*}

If $\left\{ i_1 , j_1 \right\} \cap \left\{ i_2 , j_2 \right\}$ has size 2, then $i_1 = i_2$ and $j_1 = j_2$, and hence

\[ \mathrm{E} \, \left( a_{i_1} a_{j_1} a_{i_2} a_{j_2} \right) - \frac{(b-1)^4}{16} = \frac{(7b^2-12b+5)(b^2-1)}{144} . \]

\noindent Therefore the contribution of this case in (\ref{variance2}) is

\[ \frac{(7b^2-12b+5)(b^2-1)}{144} \sum_{1 \le i < j \le m} \frac{1}{b^{2j-2i+2}} = \frac{7b^2-12b+5}{144b^2} m + O(1) . \]

\noindent Altogether we find that

\begin{equation}\label{variance3}
\mathrm{Var} \, \left( \sum_{1 \le i < j \le m} \frac{a_i a_j}{b^{j-i+1}} \right) = \frac{13b^2 -13}{144b^2} m + O(1) .
\end{equation}

\noindent Finally, it is easy to see that two times the covariance of the sums in question is

\begin{multline}\label{covariance}
2 \sum_{\substack{1 \le i_1 \le m \\ 1 \le i_2 < j_2 \le m}} \mathrm{E} \, \left( \left( \frac{(b+1) a_{i_1} - a_{i_1}^2}{2b} - \frac{b^2+3b-4}{12b} \right) \left( \frac{(b-1)^2}{4} - a_{i_2} a_{j_2} \right) \frac{1}{b^{j_2 - i_2 +1}} \right) \\ = \frac{b^3-5b^2+5b-1}{6b^2} (m-1) + O(1),
\end{multline}

\noindent by noticing that the terms for which $i_1 \not\in \left\{ i_2 , j_2 \right\}$ are all zero. Adding (\ref{variance}), (\ref{variance3}) and (\ref{covariance}), we obtain the desired formula for $\mathrm{Var} \, (S(N))$.

\end{proof}

\section{Proofs of Theorem 1 and Theorem 2}\label{section3}

Let $N$ be a random variable again, uniformly distributed in $\left\{ 0 , 1 , \dots , b^m -1 \right\}$. Proposition \ref{Proposition6} expresses $S(N)$ in terms of independent random variables $a_1 , \dots , a_m$. In this Section we prove a general central limit theorem and a large deviation result for random variables expressed in terms of independent variables in a similar way. These general results fit into the subject of weakly dependent random variables. The proof of Theorem \ref{Theorem9} below is the generalization of the proof in Section 1.3 of \cite{1}.

For positive integers $a$ and $m$ let $[m]$ denote the set $\left\{ 1 , 2 , \dots , m  \right\}$, and let

\[ {{[m]}\choose{ \le a}} = \left\{ A \subseteq [m] : |A| \le a \right\} . \]

\noindent For a finite set $A$ of integers let $\mathrm{diam} \, A = \max A - \min A$, and for random variables $X_1 , \dots , X_m$ let $X_A = \left( X_i : i \in A \right)$ for any $A \subseteq [m]$.

We are going to use the fact that for any real numbers $\lambda$ and $x$ we have

\begin{align}
\Phi (\lambda + x) &= \Phi (\lambda) + O (|x|), \label{additivephi} \\ \Phi (\lambda (1+x)) &= \Phi (\lambda ) + O (|x|) . \label{multiplicativephi}
\end{align}

\noindent Note that $\Phi (\lambda +x)-\Phi(\lambda)$ is the integral of $\frac{1}{\sqrt{2 \pi}} e^{-\frac{t^2}{2}}$ over an interval of length $|x|$, therefore (\ref{additivephi}) in fact holds with implied constant $\frac{1}{\sqrt{2 \pi}}$. Since $0\le \Phi \le 1$, (\ref{multiplicativephi}) holds for any $|x|>\frac{1}{2}$ with implied constant $2$. If $|x| \le \frac{1}{2}$, then for $\lambda \ge 0$ $\Phi(\lambda (1+x))-\Phi (\lambda)$ is an integral over an interval of length $|\lambda x|$, moreover this interval is contained in $[\lambda/2, 3 \lambda /2]$, therefore the integrand is at most $\frac{1}{\sqrt{2 \pi}} e^{-\frac{\lambda^2}{8}}$. Hence

\[ \left| \Phi (\lambda (1+x)) - \Phi (\lambda) \right| \le |\lambda x| \frac{1}{\sqrt{2 \pi}} e^{- \frac{\lambda^2}{8}} , \]

\noindent and clearly the same is true for $\lambda <0$. Note that $\frac{|\lambda|}{\sqrt{2 \pi}}e^{-\frac{\lambda^2}{8}}$ is bounded on $\mathbb{R}$, in fact the maximum is attained at $\lambda = \pm 2$ with maximum value less than 2. Thus altogether (\ref{multiplicativephi}) holds with implied constant $2$.

\begin{prop}\label{Proposition8} Let $2 \le a \le m$ be integers, and let $X_1 , X_2 , \dots , X_m$ be independent real valued random variables. For every $A \in {{[m]}\choose{ \le a}}$ let $f_A : \mathbb{R}^{|A|} \to \mathbb{R}$ be Borel measurable. Suppose that for every $A \in {{[m]}\choose{ \le a}}$ we have

\begin{enumerate}
\item[(i)] $\displaystyle{\mathrm{E} \, f_A (X_A) =0}$ ,

\item[(ii)] $\displaystyle{\left| f_A (X_A) \right| \le e^{-c \cdot \mathrm{diam} \, A}}$
\end{enumerate}

\noindent for some constant $c>0$. Let $\displaystyle{q= \left( \frac{2}{1-e^{-c}} \right)^{a+\frac{1}{2}}}$ and $\displaystyle{g(x) = \sum_{k=0}^{\infty} \frac{x^{2ak}}{(2ak)!}}$.

\begin{enumerate}

\item[(1)] For any integer $k \ge 1$ we have

\[ \mathrm{E} \, \left( \sum_{A \in {{[m]}\choose{ \le a}}} f_A (X_A) \right)^{2k} \le q^{2k} (2ak)! \cdot m^k . \]

\item[(2)] For any real number $\lambda \ge 1$ we have

\[ \Pr \left( \left| \sum_{A \in {{[m]}\choose{ \le a}}} f_A \left( X_A \right) \right| \ge \lambda q \sqrt{m} \right) \le \frac{\sqrt[a]{\lambda}}{g \left( \sqrt[a]{\lambda} -1 \right)} . \]

\end{enumerate}
\end{prop}

\begin{proof}

\noindent (1) Let $L$ denote the left hand side of the claim. By expanding we get

\begin{equation}\label{expandL}
L = \mathrm{E} \, \left( \sum_{A \in {{[m]}\choose{ \le a}}} f_A (X_A) \right)^{2k} = \sum_{A_1 , \dots , A_{2k} \in {{[m]}\choose{ \le a}}} \mathrm{E} \, \prod_{i=1}^{2k} f_{A_i} \left( X_{A_i} \right) .
\end{equation}

\noindent For each ordered $2k$-tuple $\left( A_1 , \dots , A_{2k} \right) \in {{[m]}\choose{ \le a}}^{2k}$ consider the hypergraph $\mathcal{H}$ on $[m]$ with edges $A_1 , \dots , A_{2k}$. In this proof by a hypergraph we mean an unordered collection of subsets of $[m]$, called edges, with possible repetitions. Let $p$ denote the number of connected components of $\mathcal{H}$, where $1 \le p \le 2k$. Note that if $p>k$, then there exists an isolated edge in $\mathcal{H}$, which using the independence of $X_1 , \dots , X_m$ and condition (i) implies that

\[ \mathrm{E} \, \prod_{i=1}^{2k} f_{A_i} \left( X_{A_i} \right) =0 . \]

\noindent Suppose now that $1 \le p \le k$. Let $\mathcal{C}_1 , \dots , \mathcal{C}_p$ be the connected components of $\mathcal{H}$, and let $d_j = \mathrm{diam} \, \bigcup \mathcal{C}_j$. The main observation is that the connectedness implies

\[ \mathrm{diam} \, \bigcup \mathcal{C}_j \le \sum_{A \in \mathcal{C}_j} \mathrm{diam} \, A , \]

\[ \sum_{j=1}^p d_j \le \sum_{i=1}^{2k} \mathrm{diam} \, A_i , \]

\begin{equation}\label{connected}
\left| \prod_{i=1}^{2k} f_{A_i} \left( X_{A_i} \right) \right| \le \exp \left( -c \cdot \sum_{i=1}^{2k} \mathrm{diam} \, A_i \right) \le \exp \left( -c \cdot \sum_{j=1}^p d_j \right) .
\end{equation}

Let $M_j = \min \bigcup \mathcal{C}_j$. Then $\bigcup \mathcal{C}_j \subseteq \left[ M_j , M_j + d_j \right]$. We are going to group the terms of (\ref{expandL}) according to the values $p, M_1 , \dots , M_p, d_1 , \dots , d_p$ associated with the corresponding hypergraph $\mathcal{H}$. For given $p, M_1 , \dots , M_p , d_1 , \dots , d_p$ all the sets $A_1 , \dots , A_{2k}$ have to be a subset of the set

\[ \bigcup_{j=1}^p \left[ M_j , M_j + d_j \right] \]

\noindent of size at most $\sum_{j=1}^p d_j +p$. The number of ordered $2k$-tuples \mbox{$(A_1 , \dots , A_{2k}) \in {{[m]}\choose{ \le a}}^{2k}$} for which the corresponding hypergraph $\mathcal{H}$ has associated values $p , M_1 , \dots , M_p,$ $d_1 , \dots , d_p$ is therefore at most

\[ \left( \sum_{j=1}^p d_j +p \right)^{2ak} . \]

\noindent This together with (\ref{connected}) implies that in (\ref{expandL}) we have

\[ L \le \sum_{p=1}^k \sum_{M_1 , \dots , M_p =1}^m \sum_{d_1 , \dots , d_p =0}^{\infty} \left( \sum_{j=1}^p d_j +p \right)^{2ak} \exp \left( -c \sum_{j=1}^p d_j \right) . \]

\noindent Let $d = \sum_{j=1}^p d_j$. It is known that the number of representations of a given nonnegative integer $d$ in this form is ${{d+p-1}\choose{p-1}}$, therefore we get

\[ L \le \sum_{p=1}^k \sum_{d=0}^{\infty} {{d+p-1}\choose{p-1}} (d+p)^{2ak} e^{-c d} m^p \le \sum_{p=1}^k \sum_{d=0}^{\infty} \frac{\prod_{j=1}^{2ak+p-1} (d+j)}{(p-1)!} e^{-c d} m^p . \]

\noindent The series over $d$ is in fact the well-known Taylor series

\[ \sum_{d=0}^{\infty} (d+\ell ) \cdots (d+2)(d+1) x^d = \frac{\ell !}{(1-x)^{\ell +1}} \]

\noindent with $\ell = 2ak+p-1$ and $x=e^{-c}$, thus we have

\[ L \le \sum_{p=1}^k \frac{(2ak+p-1)!}{(p-1)!} \cdot \frac{m^p}{(1-e^{-c})^{2ak+p}} = \sum_{p=1}^k {{2ak+p-1}\choose{2ak}} (2ak)! \frac{m^p}{(1-e^{-c})^{2ak+p}} . \]

\noindent Here for every $1 \le p \le k$ we have

\[ \frac{m^p}{(1-e^{-c})^{2ak+p}} \le \frac{m^k}{(1-e^{-c})^{(2a+1)k}} . \]

\noindent We can also use the combinatorial identity and trivial estimate

\[ {{n}\choose{n}} + {{n+1}\choose{n}} + \cdots + {{n+k-1}\choose{n}} = {{n+k}\choose{n+1}} \le 2^{n+k} \]

\noindent with $n=2ak$ to finally obtain

\[ L \le 2^{(2a+1)k} (2ak)! \frac{m^k}{(1-e^{-c})^{(2a+1)k}} = q^{2k} (2ak)! m^k . \]

\noindent (2) Let $P$ denote the probability in the claim. Note that $g(x)$ is monotone increasing on $\left[ 0 , \infty \right)$. Therefore for any real number $0 < \alpha < 1$ we have

\[ \begin{split} P & = \Pr \left( \left| \sum_{A \in {{[m]}\choose{ \le a}}} f_A \left( X_A \right) \right| \ge \lambda q \sqrt{m} \right) \\ & = \Pr \left( g \left( \frac{\alpha}{q^{\frac{1}{a}} m^{\frac{1}{2a}}} \left| \sum_{A \in {{[m]}\choose{ \le a}}} f_A \left( X_A \right) \right|^{\frac{1}{a}} \right) \ge g \left( \alpha \lambda^{\frac{1}{a}} \right) \right) . \end{split} \]

\noindent Applying Markov's inequality and Lebesgue's monotone convergence theorem we obtain that

\[ P \le \frac{1}{g \left( \alpha \lambda^{\frac{1}{a}} \right)} \sum_{k=0}^{\infty} \frac{\alpha^{2ak}}{q^{2k} m^k (2ak)!} \mathrm{E} \, \left( \sum_{A \in {{[m]}\choose{ \le a}}} f_A \left( X_A \right) \right)^{2k} . \]

\noindent Proposition \ref{Proposition8} (1) yields the upper bound

\[ P \le \frac{1}{g \left( \alpha \lambda^{\frac{1}{a}} \right)} \sum_{k=0}^{\infty} \alpha^{2ak} = \frac{1}{1-\alpha^{2a}} \cdot \frac{1}{g \left( \alpha \lambda^{\frac{1}{a}} \right)} . \]

\noindent Choosing $\alpha = 1 - \lambda^{- \frac{1}{a}}$ and noticing $1-\alpha^{2a} \ge 1 - \alpha = \lambda^{- \frac{1}{a}}$ finishes the proof.

\end{proof}

\begin{thm}\label{Theorem9} Let $2 \le a \le m$ be integers, and let $X_1 , X_2 , \dots , X_m$ be independent real valued random variables. For every $A \in {{[m]}\choose{ \le a}}$ let $f_A : \mathbb{R}^{|A|} \to \mathbb{R}$ be Borel measurable. Suppose that for every $A \in {{[m]}\choose{ \le a}}$ we have

\begin{enumerate}
\item[(i)] $\displaystyle{\mathrm{E} \, f_A (X_A) =0}$ ,

\item[(ii)] $\displaystyle{\left| f_A (X_A) \right| \le e^{-c \cdot \mathrm{diam} \, A}}$, 

\item[(iii)] $\displaystyle{\sigma_m^2 = \mathrm{E} \, \left( \sum_{A \in {{[m]}\choose{ \le a}}} f_A (X_A) \right)^2 > 0}$
\end{enumerate}

\noindent for some constant $c>0$. Then for any real number $\lambda$ we have

\[ \Pr \left( \frac{1}{\sigma_m} \sum_{A \in {{[m]}\choose{ \le a}}} f_A (X_A) < \lambda \right) = \Phi (\lambda ) + O \left( \sqrt[4]{\log m} \cdot \frac{m^{\frac{3}{4}}}{\sigma_m^2} \right) . \]

\noindent The implied constant in the error term depends only on $a$ and $c$.
\end{thm}

Note that Proposition \ref{Proposition8} (1) with $k=1$ implies that $\sigma_m^2 = O \left( m \right)$. The smallest attainable error term in Theorem \ref{Theorem9} is therefore $O \left( \frac{\sqrt[4]{\log m}}{\sqrt[4]{m}} \right)$, which holds whenever $\sigma_m^2 > d \cdot m$ for some constant $d>0$.

\begin{proof} Throughout this proof the implied constants in the $O$ notation will depend only on $a$ and $c$. We may assume $\sigma_m^2 \ge m^{\frac{3}{4}}$, otherwise the error term is larger than 1. We start by partitioning the set $[m]$ into $m_0$ intervals of integers $I_1 , I_2 , \dots , I_{m_0}$, in such a way that $\max I_i = \min I_{i+1} -1$ and $|I_i| = \Theta (\frac{m}{m_0})$ for any $i$. Assume $|I_i| > \frac{6}{c} \log m $ for all $i$. Let

\[ Y_i = \sum_{A \in {{I_i}\choose{ \le a }}} f_A (X_A) , \]

\[ Z_j = \sum_{\substack{A \in {{[m]}\choose{ \le a }} \\ A \cap I_j , A \cap I_{j+1} \neq \emptyset , \,\, \mathrm{diam} \, A \le \frac{3}{c} \log m}} f_A (X_A). \]

\noindent Then the random variable we are interested in can be written as

\begin{equation}\label{rv}
\sum_{A \in {{[m]}\choose{ \le a}}} f_A (X_A) = \sum_{i=1}^{m_0} Y_i + \sum_{j=1}^{m_0-1} Z_j + W,
\end{equation}

\noindent where the random variable $W$ is defined by (\ref{rv}). Then $Y_1, \dots , Y_{m_0}$ are independent, and the assumption $|I_i| > \frac{6}{c} \log m$ implies that $Z_1 , \dots , Z_{m_0-1}$ are also independent.

Since the number of sets $A \in {{[m]}\choose{ \le a }}$ such that $\mathrm{diam} \, A =d$ is at most $m \cdot (d+1)^a$, condition (ii) implies that

\begin{align}
|W| & \le \sum_{\substack{A \in {{[m]}\choose{ \le a }} \\ \mathrm{diam} \, A > \frac{3}{c} \log m}} e^{-c \cdot \mathrm{diam} \, A} \le \sum_{d > \frac{3}{c} \log m} m (d+1)^a e^{-c d} \nonumber \\ & = O \left( m \log^a m \cdot e^{-c \frac{3}{c}\log m} \right) = O \left( \frac{1}{m} \right) . \label{Wbound}
\end{align}

\noindent Similarly,

\begin{equation}\label{Ybound}
|Y_i| \le \sum_{A \in {{I_i}\choose{ \le a}}} e^{-c \cdot \mathrm{diam} \, A} \le \sum_{d=0}^{\infty} |I_i| (d+1)^a e^{-c d} = O \left( |I_i| \right) = O \left( \frac{m}{m_0} \right) .
\end{equation}

\noindent The number of sets $A \in {{[m]}\choose{ \le a }}$ with $A \subseteq \left[ \max I_j -d , \max I_j +d \right]$ is at most $(2d+1)^a$, therefore condition (ii) implies

\begin{equation}\label{Zbound}
\left| Z_j \right| \le \sum_{d=0}^{\infty} (2d+1)^a e^{-c d} = O(1).
\end{equation}

\noindent Finally, note that the number of sets $A \in {{[m]}\choose{ \le a }}$ such that $\mathrm{diam} \, A = d_1$ which intersect $\left[ \max I_j -d_2 , \max I_j +d_2 \right]$ is at most $(2d_1 + 2d_2 +1)^a$, thus from conditions (i) and (ii) we obtain that for any $i$ and $j$ we have

\begin{equation}\label{YZbound}
\left| \mathrm{E} \, \left( Y_i Z_j \right) \right| \le \sum_{d_1 , d_2 \ge 0} (2d_1 + 2d_2 +1)^a (2d_2 +1)^a e^{-c d_1} e^{-c d_2} = O(1) .
\end{equation}

\noindent By taking the variance of (\ref{rv}) we get

\begin{multline*} \sigma_m^2 = \sum_{i=1}^{m_0} \mathrm{Var} \, (Y_i) + \sum_{j=1}^{m_0-1} \mathrm{Var} \, (Z_j) + 2 \sum_{i=1}^{m_0} \sum_{j=1}^{m_0-1}  \mathrm{E} \, (Y_i Z_j) \\ + 2 \sum_{i=1}^{m_0} \mathrm{E} \, \left( Y_i W \right) + 2 \sum_{j=1}^{m_0-1} \mathrm{E} \, \left( Z_j W \right) + \mathrm{Var} \, (W) . \end{multline*}

\noindent By noticing that $\mathrm{E} \, \left( Y_i Z_j \right) =0$ unless $i=j$ or $i=j+1$, the bounds (\ref{Wbound})--(\ref{YZbound}) imply

\begin{equation}\label{sigmam}
\sigma_m^2 = \sum_{i=1}^{m_0} \mathrm{Var} \, (Y_i) + O \left( m_0 \right) .
\end{equation}

We now want to apply the Berry--Esseen theorem to the sum $\sum_{i=1}^{m_0} Y_i$ of independent random variables. Applying Proposition \ref{Proposition8} (1) with $k=2$ we obtain

\[ \mathrm{E} \, Y_i^4 = O \left( |I_i|^2 \right) = O \left( \frac{m^2}{m_0^2} \right) , \]

\noindent therefore the H\"older inequality implies

\[ \sum_{i=1}^{m_0} \mathrm{E} \, \left| Y_i \right|^3 \le \sum_{i=1}^{m_0} \left( \mathrm{E} \, Y_i^4 \right)^{\frac{3}{4}} = O \left( \frac{m^{\frac{3}{2}}}{\sqrt{m_0}} \right) . \]

\noindent As long as $m_0 = o \left( \sigma_m^2 \right)$, we can see from (\ref{sigmam}) that

\[ \left( \sum_{i=1}^{m_0} \mathrm{Var} \, \left( Y_i \right) \right)^{\frac{3}{2}} = \sigma_m^3 \left( 1 + o (1) \right) . \]

\noindent Therefore the Berry--Esseen theorem (\cite{2} Section 9.1 Theorem 3) implies that

\begin{align}
\Pr \left( \frac{1}{\sqrt{\sum_{i=1}^{m_0} \mathrm{Var} \, (Y_i)}} \sum_{i=1}^{m_0} Y_i < \lambda \right) & = \Phi (\lambda ) + O \left( \frac{\sum_{i=1}^{m_0} \mathrm{E} \, \left| Y_i \right|^3}{\left( \sum_{i=1}^{m_0} \mathrm{Var} \, \left( Y_i \right) \right)^{\frac{3}{2}}} \right) \nonumber \\ & = \Phi ( \lambda ) + O \left( \frac{m^{\frac{3}{2}}}{\sigma_m^3 \sqrt{m_0}} \right) . \label{Berry}
\end{align}

\noindent From (\ref{sigmam}) we obtain

\[ \frac{1}{\sqrt{\sum_{i=1}^{m_0} \mathrm{Var} \, (Y_i)}} = \frac{1}{\sigma_m} \cdot \left( 1 + O \left( \frac{m_0}{\sigma_m^2} \right) \right) . \]

\noindent Therefore we can use (\ref{multiplicativephi}) with $x=O \left( \frac{m_0}{\sigma_m^2} \right)$ to replace the normalizing factor in the probability in (\ref{Berry}) by $\frac{1}{\sigma_m}$ to get

\begin{equation}\label{centralY}
\Pr \left( \frac{1}{\sigma_m} \sum_{i=1}^{m_0} Y_i < \lambda \right) = \Phi (\lambda ) + O \left( \frac{m^{\frac{3}{2}}}{\sigma_m^3 \sqrt{m_0}} + \frac{m_0}{\sigma_m^2} \right) .
\end{equation}

Recall that a simple version of the Chernoff bound states that if $\zeta_1, \cdots , \zeta_n$ are independent random variables such that $\mathrm{E} \, (\zeta_j) =0$ and $|\zeta_j| \le 1$ for every $1 \le j \le n$, then for any $t>0$ we have

\[ \Pr \left( \left| \sum_{j=1}^n \zeta_j \right| > t \sqrt{n} \right) \le 2 e^{-\frac{t^2}{2}} . \]

\noindent According to (\ref{Zbound}) there exists a constant $K>0$ such that $|Z_j| \le K$ for all $j$. Condition (i) ensures that $\mathrm{E} \, (Z_j)=0$ for all $j$. Therefore we can apply the Chernoff bound to $\zeta_j=Z_j/K$ with $n=m_0-1$ and $t=\sqrt{\log m}$ to obtain

\begin{equation}\label{chernoff}
\Pr \left( \frac{1}{\sigma_m} \left| \sum_{j=1}^{m_0-1} Z_j \right| > K \sqrt{\log m} \frac{\sqrt{m_0-1}}{\sigma_m} \right) \le \frac{2}{\sqrt{m}} .
\end{equation}

\noindent From (\ref{rv}), (\ref{Wbound}) and (\ref{chernoff}) we get

\begin{multline*}
\Pr \left( \frac{1}{\sigma_m} \sum_{A \in {{[m]}\choose{ \le a}}} f_A (X_A) < \lambda \right) \\ = \Pr \left( \frac{1}{\sigma_m} \sum_{i=1}^{m_0} Y_i < \lambda + O \left( \sqrt{\log m} \frac{\sqrt{m_0}}{\sigma_m} + \frac{1}{\sigma_m m} \right) \right) + O \left( \frac{1}{\sqrt{m}} \right) .
\end{multline*}

\noindent Combining (\ref{centralY}) and (\ref{additivephi}) with $x=O \left( \sqrt{\log m} \frac{\sqrt{m_0}}{\sigma_m} + \frac{1}{\sigma_m m} \right)$ we finally obtain

\begin{multline*}
\Pr \left( \frac{1}{\sigma_m} \sum_{A \in {{[m]}\choose{ \le a}}} f_A (X_A) < \lambda \right) \\ = \Phi (\lambda ) + O \left( \frac{m^{\frac{3}{2}}}{\sigma_m^3 \sqrt{m_0}} + \frac{m_0}{\sigma_m^2} + \sqrt{\log m} \frac{\sqrt{m_0}}{\sigma_m} + \frac{1}{\sigma_m m} + \frac{1}{\sqrt{m}} \right) .
\end{multline*}

\noindent The optimal choice for $m_0$ is when the first and the third error terms are equal, which holds when

\[ m_0 = \Theta \left( \frac{m^{\frac{3}{2}}}{\sqrt{\log m} \cdot \sigma_m^2} \right) . \]

\noindent Using $\sigma_m^2 \ge m^{\frac{3}{4}}$ it is easy to check that for this choice of $m_0$ both our assumptions $|I_i| > \frac{6 \log m}{c}$ and $m_0 = o \left( \sigma_m^2 \right)$ hold.

\end{proof}

\begin{proof}[Proof of Theorem \ref{Theorem1}] First, suppose that $M=b^m$ for some integer $m \ge 2$. Let $N$ be a random variable uniformly distributed in $\left\{ 0 , 1 , \dots , b^m -1 \right\}$. Then the base $b$ digits $a_1 , \dots , a_m$ of $N$ are independent random variables. Let $K>0$ be a constant for which

\[ \left| \frac{(b+1)a_i - a_i^2}{2b} - \mathrm{E} \, \frac{(b+1)a_i - a_i^2}{2b} \right| \le K b , \]

\[ \left| \frac{a_i a_j}{b} - \mathrm{E} \, \left( \frac{a_i a_j}{b} \right) \right| \le K b \]

\noindent for any $1 \le i < j \le m$. Using Proposition \ref{Proposition6} we can write $S(N)$ in the form

\[ S(N) - \mathrm{E} \, \left( S(N) \right) = K b \sum_{A \in {{[m]}\choose{\le 2}}} f_A (a_A) , \]

\noindent where $f_{\emptyset} =0$, $f_{\left\{ i \right\}} \left( x \right) = \frac{(b+1)x - x^2}{2Kb^2} - \mathrm{E} \, \frac{(b+1)a_i - a_i^2}{2Kb^2}$ and for $1 \le i < j \le m$

\[ f_{\left\{ i ,j \right\}} \left( x,y \right) = - \left( \frac{x y}{Kb^2} - \mathrm{E} \, \left( \frac{a_i a_j}{Kb^2} \right) \right) \cdot \frac{1}{b^{j-i}} . \]

\noindent Then the conditions of Theorem \ref{Theorem9} are satisfied with $a=2$ and $c= \log 2$. According to Proposition \ref{Proposition7} we have $\sigma_m^2 = \frac{1}{K^2 b^2} \mathrm{Var} \, \left( S(N) \right) = \Theta (m)$, hence we obtain

\[ \Pr \left( \frac{S(N) - \mathrm{E} \, \left( S(N) \right)}{\sqrt{\mathrm{Var} \, \left( S(N) \right)}} < \lambda \right) = \Phi (\lambda ) + O \left( \frac{\sqrt[4]{\log m}}{\sqrt[4]{m}} \right) . \]

\noindent Since $d(b)=\Theta (b^2)$, from Proposition \ref{Proposition7} we can see that

\begin{align*}
\frac{1}{\sqrt{\mathrm{Var} \, \left( S(N) \right)}} &= \frac{1}{\sqrt{d(b)m}} \left( 1 + O \left( \frac{1}{bm} \right) \right) , \\
\frac{\mathrm{E} \, \left( S(N) \right)}{\sqrt{d(b)m}} &= \frac{c(b)m}{\sqrt{d(b)m}} + O \left( \frac{1}{b \sqrt{m}} \right) .
\end{align*}

\noindent Hence if we replace $\mathrm{Var} \, \left( S(N) \right)$ by $d(b)m$, and then $\mathrm{E} \, \left( S(N) \right)$ by $c(b)m$ in the probability, then using (\ref{multiplicativephi}) with $x=O \left( \frac{1}{bm} \right)$ and (\ref{additivephi}) with $x=O \left( \frac{1}{b \sqrt{m}} \right)$ the error we make is $O \left( \frac{1}{bm} + \frac{1}{b \sqrt{m}} \right)$. Thus

\begin{equation}\label{SNdistribution}
\Pr \left( \frac{S(N) - c(b) m}{\sqrt{d(b) m}} < \lambda \right) = \Phi (\lambda ) + O \left( \frac{\sqrt[4]{\log m}}{\sqrt[4]{m}} \right) .
\end{equation}

We now show that (\ref{SNdistribution}) holds for any $M>b^2$. Let $M = \sum_{i=1}^{m} c_i b^{i-1}$ be the base $b$ representation of $M$, where $c_i \in \left\{ 0 , 1, \dots , b-1 \right\}$ and $c_m >0$. Let

\[ M^* = \sum_{m -\log m -1 \le i \le m} c_i b^{i-1} . \]

\noindent Let $N$ be a random variable uniformly distributed in $\left\{ 0 , 1 , \dots , M^* -1 \right\}$, and consider its base $b$ representation $N = \sum_{i=1}^m a_i b^{i-1}$. Note that we allow $a_m$ to be zero. Then the random variables $\left( a_i : 1 \le i < m - \log m -1 \right)$ are independent, and each is uniformly distributed in $\left\{ 0 , 1 , \dots , b-1 \right\}$. Let us introduce new random variables $a_j^*$ for every $m - \log m -1 \le j \le m$, such that

\[ \left( a_i , a_j^* : 1 \le i < m - \log m -1 \le j \le m \right) \]

\noindent are identically distributed independent random variables. Let

\[ N^* = \sum_{1 \le i < m - \log m -1} a_i b^{i-1} + \sum_{m - \log m -1 \le j \le m} a_j^* b^{j-1} . \]

\noindent Then $S\left( N^* \right)$ satisfies (\ref{SNdistribution}). Note that there are $O(\log m)$ base $b$ digits at which $N$ and $N^*$ differ. According to the formula in Proposition \ref{Proposition6}, if a single base $b$ digit of $N$ is changed, $S(N)$ can change by at most $O(b)$. Hence $S(N^*)=S(N)+O(b \log m)$. Using (\ref{additivephi}) with $x=O \left( \frac{\log m}{\sqrt{m}} \right)$, the error of replacing $S(N^*)$ in (\ref{SNdistribution}) by $S(N)$ is $O \left( \frac{\log m}{\sqrt{m}} \right)$, therefore

\[ \frac{1}{M^*} \left| \left\{ 0 \le N < M^* : \frac{S(N) - c(b) m}{\sqrt{d(b) m}} < \lambda  \right\} \right| = \Phi (\lambda ) + O \left( \frac{\sqrt[4]{\log m}}{\sqrt[4]{m}} \right) . \]

\noindent Here the error of replacing $M^*$ by $M$ is

\[ O \left( \frac{M-M^*}{M} \right) = O \left( \frac{b^{m- \log m -1}}{b^{m-1}} \right) = O \left( \frac{\sqrt[4]{\log m}}{\sqrt[4]{m}} \right) . \]

\noindent Finally, note that $\frac{M}{m} \le N \le M$ with probability $1 - O \left( \frac{1}{m} \right)$, and for all such $N$ we have $\log_b N = m + O \left( \log m \right)$. Using (\ref{additivephi}) with $x=O \left( \frac{\log m}{\sqrt{m}} \right)$, the error of replacing $c(b)m$ by $c(b) \log_b N$ is $O \left( \frac{\log m}{\sqrt{m}} \right)$. Using (\ref{multiplicativephi}) with $x=O \left( \frac{\log m}{m} \right)$, the error of replacing $\sqrt{d(b)m}$ by $\sqrt{d(b) \log_b N}$ is $O \left( \frac{\log m}{m} \right)$. Hence we get

\[ \frac{1}{M} \left| \left\{ 0 \le N < M : \frac{S(N) - c(b) \log_b N}{\sqrt{d(b) \log_b N}} < \lambda  \right\} \right| = \Phi (\lambda ) + O \left( \frac{\sqrt[4]{\log m}}{\sqrt[4]{m}} \right) . \]

\noindent The error term can be expressed in terms of $M$ by noting $m \ge \log_b M $.

\end{proof}

\begin{proof}[Proof of Theorem \ref{Theorem2}] First, assume $M = b^m$ for some integer $m \ge 2$. Let $N$ be a random variable uniformly distributed in $\left\{ 0 , 1 , \dots , b^m -1 \right\}$, and let $N= \sum_{i=1}^m a_i b^{i-1}$ be the base $b$ representation of $N$, where $a_1 , \dots , a_m$ are independent random variables, each uniformly distributed in $\left\{ 0 , 1 , \dots , b-1 \right\}$. Note that for any $1 \le i < j \le m$ we have

\[ \left| \frac{(b+1)a_i - a_i^2}{2b} - \frac{(b+1) \mathrm{E} \, (a_i) - \mathrm{E} \, (a_i^2 )}{2b}  \right| \le \frac{3}{4} b , \]

\[ \left| \frac{a_i a_j}{b} - \frac{\mathrm{E} \, (a_i) \mathrm{E} \, (a_j)}{b} \right| \le \frac{3}{4} b . \]

\noindent Using Proposition \ref{Proposition6} we can write $S(N)$ in the form

\[ S(N) - \mathrm{E} \, \left( S(N) \right) = \frac{3}{4} b \sum_{A \in {{[m]}\choose{\le 2}}} f_A \left( a_A \right) , \]

\noindent where $f_{\emptyset} =0$, $f_{\left\{ i \right\}} \left( x \right) = \frac{4}{3b} \frac{(b+1)x - x^2}{2b} - \frac{4}{3b} \mathrm{E} \, \frac{(b+1)a_i - a_i^2}{2b}$ and for $1 \le i < j \le m$

\[ f_{\left\{ i ,j \right\}} \left( x,y \right) = - \frac{4}{3b} \left( \frac{x y}{b} - \mathrm{E} \, \left( \frac{a_i a_j}{b} \right) \right) \cdot \frac{1}{b^{j-i}} . \]

\noindent Then the conditions of Proposition \ref{Proposition8} (2) are satisfied with $a=2$, $c= \log 2$, $q = 32$ and

\[ g(x) = \sum_{k=0}^{\infty} \frac{x^{4k}}{(4k)!} = \frac{e^{x} + e^{-x}}{4} + \frac{\cos x}{2} \ge \frac{e^x-2}{4} . \]

\noindent Therefore Proposition \ref{Proposition8} (2) yields

\begin{align}
\Pr \left( \left| S(N) - \mathrm{E} \, \left( S(N) \right) \right| \ge 24 \lambda b \sqrt{m} \right) & = \Pr \left( \left| \sum_{A \in {{[m]}\choose{ \le 2}}} f_A \left( a_A \right) \right| \ge 32 \lambda \sqrt{m} \right) \nonumber \\ & \le \frac{4 \sqrt{\lambda}}{e^{\sqrt{\lambda} -1}-2} . \label{largedev}
\end{align}

Now we prove (\ref{largedev}) holds for any integer $M > b$. Let $M = \sum_{i=1}^{m} c_i b^{i-1}$ be the base $b$ representation of $M$, where $c_i \in \left\{ 0 , 1, \dots , b-1 \right\}$ and $c_m >0$. Let

\[ M^* = \sum_{m - \sqrt{m} + 1 \le i \le m} c_i b^{i-1} . \]

\noindent Let $N$ be a random variable uniformly distributed in $\left\{ 0 , 1 , \dots , M^* -1 \right\}$, and consider its base $b$ representation $N = \sum_{i=1}^m a_i b^{i-1}$. Then $\left( a_i : 1 \le i < m - \sqrt{m}+1 \right)$ are independent random variables, each uniformly distributed in $\left\{ 0 , 1 , \dots , b-1 \right\}$. Let us introduce new random variables $a_j^*$ for $m - \sqrt{m}+1 \le j \le m$ such that

\[ \left( a_i , a_j^* : 1 \le i < m - \sqrt{m} +1 \le j \le m \right) \]

\noindent are identically distributed independent random variables. Let

\[ N^* = \sum_{1 \le i < m - \sqrt{m}+1} a_i b^{i-1} + \sum_{m - \sqrt{m} +1 \le j \le m} a_j^* b^{j-1} . \]

\noindent Then $S \left( N^* \right)$ satisfies (\ref{largedev}). Using Proposition \ref{Proposition6} and Proposition \ref{Proposition7} we get the following estimates:

\begin{align*}
\left| \mathrm{E} \, \left( S(N^*) \right) - \frac{b^2-1}{12b} m \right| & \le \frac{1}{4} \le \frac{\lambda b \sqrt{m}}{24 \sqrt{2}} , \\
\left| \frac{b^2-1}{12b} m - \frac{b^2-1}{12b} \log_b M \right| & \le \frac{b^2-1}{12b} \le \frac{\lambda b \sqrt{m}}{36 \sqrt{2}} , \\
\left| S(N) - S(N^*) \right| & \le \frac{(b+1)^2}{8 b} \sqrt{m} + 2 \sqrt{m} \le \frac{41}{96} \lambda b \sqrt{m} .
\end{align*}

\noindent Since

\[ 24 + \frac{1}{24 \sqrt{2}} + \frac{1}{36 \sqrt{2}} + \frac{41}{96} < 25 , \]

\noindent these estimates imply

\[ \frac{1}{M^*} \left| \left\{ 0 \le N < M^* : \left| S(N) - \frac{b^2-1}{12b} \log_b M \right| \ge 25 \lambda b \sqrt{\log_b M +1} \right\} \right| \le \frac{4 \sqrt{\lambda}}{e^{\sqrt{\lambda}-1}-2} . \]

\noindent Finally, note that the error of replacing $M^*$ by $M$ is at most

\[ \frac{M-M^*}{M} \le \frac{b^{m-\sqrt{m}+1}}{b^{m-1}} \le \frac{1}{b^{\sqrt{\log_b M}-2}} . \]

\end{proof}

\section{Proofs of Theorem 3 and Theorem 4}\label{section4}

In this Section the proofs of Theorem \ref{Theorem3},  Theorem \ref{Theorem4} and Proposition \ref{Proposition5} are given. We start by estimating an exponential sum in terms of the base $b$ van der Corput sequence. Proposition \ref{Proposition10} below is a special case of Lemma 3 in \cite{6}. Nevertheless, for the sake of completeness a proof is included.

\begin{prop}\label{Proposition10} Let $b \ge 2$ be an integer and let $x_n$ denote the base $b$ van der Corput sequence. If $\ell$ is an integer such that $b^s \nmid \ell$ for some positive integer $s$, then for any positive integer $N$ we have

\[ \left| \sum_{n=0}^{N-1} e^{2 \pi i \ell x_n} \right| < b^s . \]

\end{prop}

\begin{proof} Let $N = \sum_{j=1}^m a_j b^{j-1}$ be the base $b$ representation of $N$ with base $b$ digits $a_j \in \left\{ 0 , 1 , \dots , b-1 \right\}$ with $a_m \neq 0$. By splitting the sum we get

\begin{equation}\label{splitsum}
\left| \sum_{n=0}^{N-1} e^{2 \pi i \ell x_n} \right| \le \left| \sum_{n=0}^{a_m b^{m-1}-1} e^{2 \pi i \ell x_n} \right| + \left| \sum_{n=a_m b^{m-1}}^{N-1} e^{2 \pi i \ell x_n} \right| .
\end{equation}

\noindent Note that for any $a_m b^{m-1} \le n < N$ the base $b$ representation of $n$ starts with the digit $a_m$. From the definition of the base $b$ van der Corput sequence we know that for any such $n$ we have $x_n = x_{n-a_m b^{m-1}} + \frac{a_m}{b^m}$, therefore we can reindex the second sum to obtain

\[ \left| \sum_{n=a_m b^{m-1}}^{N-1} e^{2 \pi i \ell x_n} \right| = \left| e^{2 \pi i \ell \frac{a_m}{b^m}} \sum_{n=0}^{N-a_m b^{m-1}-1} e^{2 \pi i \ell x_n} \right| . \]

\noindent Using the base $b$ representation of $N$, repeated application of the triangle inequality in (\ref{splitsum}) yields

\begin{equation}\label{upperbound}
\left| \sum_{n=0}^{N-1} e^{2 \pi i \ell x_n} \right| \le \sum_{j=1}^m \left| \sum_{n=0}^{a_j b^{j-1}-1} e^{2 \pi i \ell x_n} \right| .
\end{equation}

\noindent For any $1 \le j \le m$ we have

\[ \left\{ x_n : 0 \le n < a_j b^{j-1} \right\} = \left\{ \frac{k}{b^{j-1}} + \frac{a}{b^j} : 0 \le k < b^{j-1} , \quad 0 \le a < a_j \right\} , \]

\noindent therefore

\[ \left| \sum_{n=0}^{a_j b^{j-1}-1} e^{2 \pi i \ell x_n} \right| = \left| \sum_{k=0}^{b^{j-1}-1} e^{2 \pi i \frac{\ell}{b^{j-1}} k} \right| \cdot \left| \sum_{a=0}^{a_j-1} e^{2 \pi i \ell \frac{a}{b^{j}}} \right| . \]

\noindent The assumption $b^s \nmid \ell$ implies that the first factor is zero whenever $s \le j-1$. Thus we get from (\ref{upperbound}) that

\[ \left| \sum_{n=0}^{N-1} e^{2 \pi i \ell x_n} \right| \le \sum_{j=1}^s \left| \sum_{n=0}^{a_j b^{j-1}-1} e^{2 \pi i \ell x_n} \right| \le \sum_{j=1}^s a_j b^{j-1} < b^s . \]

\end{proof}

\begin{proof}[Proof of Theorem \ref{Theorem3}] It is enough to prove the theorem in the special case when $p$ is a positive even integer. Indeed, if $p \ge 1$ is arbitrary, we can choose a positive even integer $p' > p$. Observation (\ref{SN}) then implies

\[ S(N) \le \left\| \Delta_N \right\|_p \le \left\| \Delta_N \right\|_{p'} . \]

\noindent Theorem \ref{Theorem1} and Theorem \ref{Theorem3} with $p'$ thus imply Theorem \ref{Theorem3} with $p$.

From now on we assume $p$ is a positive even integer. Every implied constant in the $O$ notation will depend only on $p$. From the alternative form of the discrepancy function

\[ \Delta_N (x) = \sum_{n=0}^{N-1} \left( \chi_{(x_n , 1 ]} (x) -x \right) , \]

\noindent where $\chi$ denotes the characteristic function of a set, one obtains via routine integration that for any integer $\ell \neq 0$ we have

\[ \int_0^1 \Delta_N (x) e^{- 2 \pi i \ell x} \, \mathrm{d}x = \frac{1}{2 \pi i \ell} \sum_{n=0}^{N-1} e^{- 2 \pi i \ell x_n} . \]

\noindent Therefore Parseval's formula and observation (\ref{SN}) yield

\[ \int_0^1 \left( \Delta_N (x) - S(N) \right)^2 \, \mathrm{d}x = \sum_{\ell \neq 0} \frac{1}{4 \pi^2 \ell^2} \left| \sum_{n=0}^{N-1} e^{- 2 \pi i \ell x_n} \right|^2 . \]

\noindent Let $N = \sum_{j=1}^m a_j b^{j-1}$ be the base $b$ representation of $N$, where $a_j \in \left\{ 0 , 1 , \dots , b-1 \right\}$ and $a_m >0$. Note $N < b^m$. Let $b^s \parallel \ell$ denote the fact that $b^s \mid \ell$ but $b^{s+1} \nmid \ell$. By splitting the sum according to the highest power of $b$ dividing $\ell$, and applying Proposition \ref{Proposition10} and a trivial estimate we obtain

\begin{align}
& \int_0^1 \left( \Delta_N (x) - S(N) \right)^2 \, \mathrm{d}x \nonumber \\ & \hspace{30mm} = \sum_{s=0}^{m-2} \sum_{\substack{\ell \neq 0 \\ b^s \parallel \ell}} \frac{1}{4 \pi^2 \ell^2} \left| \sum_{n=0}^{N-1} e^{- 2 \pi i \ell x_n} \right|^2 + \sum_{\substack{\ell \neq 0 \\ b^{m-1} \mid \ell}} \frac{1}{4 \pi^2 \ell^2} \left| \sum_{n=0}^{N-1} e^{- 2 \pi i \ell x_n} \right|^2 \nonumber \\ & \hspace{30mm} \le \sum_{s=0}^{m-2} \sum_{\substack{\ell \neq 0 \\ b^s \parallel \ell}} \frac{1}{4 \pi^2 \ell^2} b^{2s+2} + \sum_{\substack{\ell \neq 0 \\ b^{m-1} \mid \ell}} \frac{1}{4 \pi^2 \ell^2} b^{2m} \nonumber \\ & \hspace{30mm} \le \sum_{s=0}^{m-1} \sum_{t \neq 0} \frac{b^2}{4 \pi^2 t^2} = \frac{b^2}{12} m \le \frac{b^2}{12} \left( \log_b N +1 \right) . \label{L2norm}
\end{align}

For a positive even integer $p$ consider the binomial formula

\begin{multline*}
\Delta_N (x)^p = S(N)^p + p S(N)^{p-1} \left( \Delta_N (x) - S(N) \right) \\ + \sum_{k=2}^p {{p}\choose{k}} S(N)^{p-k} \left( \Delta_N (x) - S(N) \right)^k .
\end{multline*}

\noindent By integrating on $[0,1]$ we get

\[ \int_0^1 \Delta_N (x)^p \, \mathrm{d}x = S(N)^p + \sum_{k=2}^p {{p}\choose{k}} S(N)^{p-k} \int_0^1 \left( \Delta_N (x) - S(N) \right)^k \, \mathrm{d} x . \]

\noindent Using the facts that $\Delta_N (x) = O \left( b \left( \log_b N +1 \right) \right)$ and $S(N) = O \left( b \left( \log_b N +1 \right) \right)$, we get from (\ref{L2norm}) that for any $2 \le k \le p$

\[ \begin{split} \int_0^1 \left( \Delta_N (x) - S(N) \right)^k \, \mathrm{d}x & \le \sup_{x \in [0,1]} \left| \Delta_N (x) - S(N) \right|^{k-2} \int_0^1 \left( \Delta_N (x) - S(N) \right)^2 \, \mathrm{d} x \\ &= O \left( b^k \left( \log_b N +1 \right)^{k-1} \right) . \end{split}\]

\noindent Thus we have

\begin{equation}\label{Lpnorm}
\int_0^1 \Delta_N (x)^p \, \mathrm{d}x = S(N)^p + O \left( b^p \left( \log_b N +1 \right)^{p-1}  \right) .
\end{equation}

Now we prove the theorem. Let $M > b^2$, and let $N$ be a random variable uniformly distributed in $\left\{ 0 , 1 , \dots , M-1 \right\}$. We know from Theorem \ref{Theorem1} that the event

\[ \frac{S(N)-c(b) \log_b N}{\sqrt{d(b) \log_b N}} > - \frac{c(b)}{4\sqrt{d(b)}} \sqrt{\log_b M} \]

\noindent has probability

\[ 1 - \Phi \left( - \frac{c(b)}{4 \sqrt{d(b)}} \sqrt{\log_b M} \right) - O \left( \frac{\sqrt[4]{\log \log_b M}}{\sqrt[4]{\log_b M}} \right) = 1- O \left( \frac{\sqrt[4]{\log \log_b M}}{\sqrt[4]{\log_b M}} \right) . \]

\noindent The event $M^{3/4} \le N < M$ also has probability

\[ 1 - O \left( \frac{1}{\sqrt[4]{M}} \right) = 1 - O \left( \frac{\sqrt[4]{\log \log_b M}}{\sqrt[4]{\log_b M}} \right) . \]

\noindent Therefore it is enough to consider the intersection of these two events, on which

\[ \begin{split} S(N) & > c(b) \left( \log_b N - \frac{1}{4} \sqrt{\log_b N \log_b M} \right) \\ & \ge c(b) \left( \frac{3}{4} \log_b M - \frac{1}{4} \sqrt{\log_b M \log_b M} \right) = \frac{1}{2} c(b) \log_b M \end{split} \]

\noindent holds. For every such $N$ we get from (\ref{Lpnorm}) that

\begin{align*}
\int_0^1 \Delta_N (x)^p \, \mathrm{d}x &= S(N)^p \left( 1 + O \left( \frac{1}{\log_b M} \right) \right) , \\ \left\| \Delta_N \right\|_p &= S(N) \left( 1 + O \left( \frac{1}{\log_b M} \right) \right) = S(N) + O \left( b \right) .
\end{align*}

\noindent Theorem \ref{Theorem3} is thus reduced to Theorem \ref{Theorem1}.

\end{proof}

\begin{proof}[Proof of Theorem \ref{Theorem4}] Similarly to the proof of Theorem \ref{Theorem3} we may assume that $p$ is a positive even integer. Since $\left\| \Delta_N \right\|_p = O \left( b \left( \log_b N + 1 \right) \right)$, by choosing $A$ large enough we may assume that $\lambda \le \sqrt{\log_b M}$. Recall (\ref{Lpnorm}) from the proof of Theorem \ref{Theorem3}:

\[ \int_0^1 \Delta_N (x)^p \, \mathrm{d}x = S(N)^p + O \left( b^p \left( \log_b N +1 \right)^{p-1}  \right) \]

\noindent for any $N>0$. Let $N$ be a random variable which is uniformly distributed in $\left\{ 0 , 1 , \dots , M-1 \right\}$. We know from Theorem \ref{Theorem2} that $S(N) > \frac{1}{2} c(b) \log_b M$ with probability

\[ 1 - O \left( e^{-c \sqrt[4]{\log_b M}} + \frac{1}{b^{\sqrt{\log_b M}-2}} \right) \]

\noindent for some constant $c>0$. We also have $\frac{M}{b^{\sqrt{\log_b M}}} \le N < M$ with probability at least $1-O \left( \frac{1}{b^{\sqrt{\log_b M}}} \right)$. For all such $N$ we have $\left\| \Delta_N \right\|_p = S(N) + O(b)$ and

\begin{align*}
\log_b N &= \log_b M + O \left( \sqrt{\log_b M} \right) , \\ \sqrt{\log_b N} &= \sqrt{\log_b M} + O(1) .
\end{align*}

\noindent These estimates together with Theorem \ref{Theorem2} yield

\begin{multline*}
\frac{1}{M} \left| \left\{ 0 \le N < M : \left| \left\| \Delta_N \right\|_p - \frac{b^2-1}{12 b} \log_b N \right| \ge A \lambda b \sqrt{\log_b N} \right\} \right| \\ = O \left( \frac{\sqrt{\lambda}}{e^{\sqrt{\lambda}-1}-2} + e^{-c \sqrt[4]{\log_b M}} + \frac{1}{b^{\sqrt{\log_b M}-2}} \right)
\end{multline*}

\noindent for any $\lambda \ge 3$ with some constant $A>0$ depending only on $p$. By replacing $A$ by a larger constant we can simplify the upper bound to $e^{-\sqrt{\lambda}}$ and relax the condition $\lambda \ge 3$ to $\lambda \ge 1$.

\end{proof}

\begin{proof}[Proof of Proposition \ref{Proposition5}] Let us write $f(x)$ in the form

\begin{equation}\label{Bernoulli}
f(x) = \int_0^1 f(t) \, \mathrm{d}t + \left( f(1) - f(0) \right) \left( x - \frac{1}{2} \right) + g(x) ,
\end{equation}

\noindent where $g : [0,1] \to \mathbb{R}$ is defined via (\ref{Bernoulli}). Then we have $\int_0^1 g(x) \, \mathrm{d}x = 0$ and $g(0)=g(1)$. Note that (\ref{Bernoulli}) is the expansion of $f(x)$ with respect to the Bernoulli polynomials with an explicit remainder term. For any integer $N>0$ we have

\[ \sum_{n=0}^{N-1} f(x_n) = N \int_0^1 f(t) \, \mathrm{d}t - \left( f(1) - f(0) \right) S(N) + \sum_{n=0}^{N-1} g(x_n) . \]

We now have to show that the last sum is negligible. Since $g$ is twice differentiable on [0,1] and $g(0)=g(1)$, we have that the periodic extension of $g$ to $\mathbb{R}$ with period 1 is Lipschitz, therefore its Fourier series converges to $g$:

\[ g(x) = \sum_{\ell \in \mathbb{Z}} \hat{g} (\ell ) e^{2 \pi i \ell x} \]

\noindent for any $x \in [0,1]$, where

\[ \hat{g} (\ell ) = \int_{0}^1 g(x) e^{- 2 \pi i \ell x} \, \mathrm{d} x . \]

\noindent We have $\hat{g} (0)=0$, because $\int_0^1 g(x) \, \mathrm{d}x =0$. Since $g(0)=g(1)$, integration by parts yields that for any integer $\ell \neq 0$

\[ \hat{g} (\ell ) = \frac{g'(1)-g'(0)}{4 \pi^2 \ell^2} - \int_0^1 g''(x) \frac{e^{- 2 \pi i \ell x}}{4 \pi^2 \ell^2} \, \mathrm{d}x = \frac{1}{4 \pi^2 \ell^2} \int_0^1 g''(x) \left( 1 - e^{- 2 \pi i \ell x} \right) \, \mathrm{d}x , \]

\[ \left| \hat{g} (\ell ) \right| \le \frac{1}{2 \pi^2 \ell^2} \int_0^1 \left| g'' (x) \right| \, \mathrm{d}x = \frac{\left\| f'' \right\|_1}{2 \pi^2 \ell^2} .  \]

\noindent Therefore

\[ \left| \sum_{n=0}^{N-1} g(x_n) \right| = \left| \sum_{n=0}^{N-1} \sum_{\ell \neq 0} \hat{g} ( \ell) e^{2 \pi i \ell x_n} \right| \le \sum_{\ell \neq 0} \frac{\left\| f'' \right\|_1}{2 \pi^2 \ell^2} \left| \sum_{n=0}^{N-1} e^{2 \pi i \ell x_n} \right| . \]

\noindent We can split up the sum according to the highest power of $b$ dividing $\ell$. Proposition \ref{Proposition10} hence gives

\begin{align*}
\left| \sum_{n=0}^{N-1} g(x_n) \right| & \le \sum_{s=0}^{\infty} \sum_{\substack{\ell \neq 0 \\ b^s \parallel \ell}} \frac{\left\| f'' \right\|_1}{2 \pi^2 \ell^2}  \left| \sum_{n=0}^{N-1} e^{2 \pi i \ell x_n} \right| \\ & \le \sum_{s=0}^{\infty} \sum_{t \neq 0} \frac{\left\| f'' \right\|_1}{2 \pi^2 b^{2s} t^2} b^{s+1} = \frac{b^2}{6(b-1)} \left\| f'' \right\|_1 \le \frac{b}{3} \left\| f'' \right\|_1 .
\end{align*}

\end{proof}

\end{document}